\documentclass[11pt, leqno,twoside]{article}
\usepackage{amssymb}
\usepackage{amsmath}
\usepackage{amsthm}
\usepackage{graphicx}

\usepackage[active]{srcltx}

\allowdisplaybreaks

\def\rr{{\mathbb R}}
\def\rn{{{\rr}^n}} 
\def\zz{{\mathbb Z}}
\def\cb{{\mathcal B}}
\def\cd{{\mathcal D}}
\def\cf{{\mathcal F}}
\def\cm{{\mathcal M}}

\def\gz{{\gamma}}

\def\boz{{\Omega}} 
\def\sz{\sigma}
\def\fz{\infty}
\def\wz{\widetilde}
\def\ls{\lesssim}
\def\gs{\gtrsim}
\def\r{\right}
\def\lf{\left}
\def\dist{{\mathop\mathrm{\,dist\,}}}

\def\ball{{\mathop\mathrm{\,ball\,}}}

\def\llc{{\mathop\mathrm{\,LLC}}}

\def\lp{{L^{p}(\boz)}}
\def\diam{{\mathop\mathrm{\,diam\,}}}

\def\bint{{\ifinner\rlap{\bf\kern.35em--}
\int\else\rlap{\bf\kern.45em--}\int\fi}\ignorespaces}

\def\bbint{{\ifinner\rlap{\bf\kern.35em--}
\hspace{0.078cm}\int\else\rlap{\bf\kern.45em--}\int\fi}\ignorespaces}

\newtheorem{thm}{Theorem}[section]
\newtheorem{lem}{Lemma}[section]

\newtheorem{cor}{Corollary}[section]
\newtheorem{defn}{Definition}[section]
\numberwithin{equation}{section}

\begin{document}

\arraycolsep=1pt

\title{\Large\bf Criteria for Optimal Global Integrability of Haj\l asz-Sobolev Functions
\footnotetext{\hspace{-0.35cm}
\noindent{2000 {\it Mathematics Subject Classification:}} 46E35
\endgraf  {\it Key words and phases:}  John domain, weak carrot domain, local linear connectivity, Haj\l asz-Sobolev space,
Haj\l asz-Sobolev-Poincar\'e imbedding, Haj\l asz-Trudinger imbedding, Haj\l asz-Sobolev extension
\endgraf The author was supported by the Academy of Finland grant 120972.
}}
\author{Yuan Zhou}
\date{ }
\maketitle

\begin{center}
\begin{minipage}{13.5cm}\small
{\noindent{\bf Abstract}\quad
The author establishes some geometric criteria for a domain of $\rn$ with $n\ge2$
to support a $(pn/(n-ps),\,p)_s$-Haj\l asz-Sobolev-Poincar\'e imbedding with $s\in(0,\,1]$ and $p\in(n/(n+s),\,n/s)$
or an $s$-Haj\l asz-Trudinger imbedding with $s\in(0,\,1]$.
}
\end{minipage}
\end{center}

\medskip

\section{Introduction\label{s1}}

\hskip\parindent

The study of the Haj\l asz spaces $\dot M^{1,\,p}$
 was initiated by Haj\l asz \cite{h96} on arbitrary metric measure spaces,
see \cite{h96,hk00, h03,y03,ks08,kyz09,kyz09c}
for further discussions, generalizations and connections
 with the classical (Hardy-)Sobolev, Besov and Triebel-Lizorkin spaces.
In particular, a fractional version $\dot M^{s,\,p}$ with  $s\in(0,\,1)$
was introduced by  Yang \cite{y03},
and a Sobolev-type version $\dot M^{1,\,p}_\ball$ on  domains by Koskela and Saksman \cite{ks08}.

We first recall some definitions and notions.
In this paper, we always let $n\ge 2$ and $\boz$ be a domain of $\rn$.
For every $s\in (0,1]$ and measurable function
$u$, denote by $\cd^s(u)$ the collection of all nonnegative measurable functions $g$
such that
\begin{equation}\label{e1.1}
|u(x)-u(y)|\le |x-y|^s[g(x)+g(y)]
\end{equation}
for all $x,\ y\in\boz\setminus E$, where
$E\subset\boz$ with $|E|=0$.
We also denote by $\cd_\ball^s(u)$ the collection of
all nonnegative measurable functions $g$ such that
 \eqref{e1.1} holds for all $x,\ y\in\boz\setminus E$
satisfying $|x-y|<\frac12\dist(x,\,\partial\boz)$.

\begin{defn}\rm\label{d1.1}
Let $s\in (0,1]$ and $p\in(0,\,\fz)$.
Then the homogeneous Haj\l asz space
 $\dot M^{s,\,p}(\boz)$ is the
space of all measurable functions $u$ such that
$$\|u\|_{\dot M^{s,\,p} (\boz)}\equiv\inf_{g\in \cd^s (u)}\|g\|_\lp<\fz,$$
and its  Sobolev-type version $\dot M^{s,\,p}_\ball(\boz)$ is the
space of all measurable functions $u$ such that
$$\|u\|_{\dot M^{s,\,p}_\ball(\boz)}\equiv\inf_{g\in \cd^s_\ball(u)}\|g\|_\lp<\fz.$$
\end{defn}

Obviously, for all $s\in(0,\,1]$ and $p\in(0,\,\fz)$,
$\dot M^{s,\,p}(\boz)\subset \dot M^{s,\,p}_\ball(\boz)$.
If $\boz$ is a uniform domain,
then $\dot M^{s,\,p}_\ball(\boz)=\dot M^{s,\,p}(\boz)$ for all $s\in(0,\,1]$ and $p\in(n/(n+s),\,\fz)$;
 see \cite[Theorem 19]{ks08}. But, generally, we cannot expect that $\dot M^{s,\,p}(\boz)=\dot M^{s,\,p}_\ball(\boz)$.
For example,  this fails when $\boz=B(0,\,1)\setminus \{(x,\,0):\,x\ge0\}\subset\rr^2.$

 Haj\l asz-Sobolev spaces are closely related to the classical (Hardy-)Sobolev and Triebel-Lizorkin spaces.
In this paper, we always denote by $\dot W^{1,\,p}(\boz)$ with $p\in(1,\,\fz)$ the homogeneous Sobolev space,
by $\dot H^{1,\,p}(\boz)$ with $p\in(0,\,1]$ the Hardy-Sobolev space as in \cite{m90,m91},
and by  $F^s_{p,\,q}(\rn)$ with $s\in\rr$ and $p,\,q\in(0,\,\fz]$
the homogeneous Triebel-Lizorkin spaces as in \cite{t83}.
It was proved in \cite{h96,ks08} that
$\dot W^{1,\,p}(\boz)=\dot M^{1,\,p}_\ball(\boz)$ for $p\in(1,\,\fz)$
 and  $\dot H^{1,\,p}(\boz)=\dot M^{1,\,p}_\ball(\boz)$ for $p\in(n/(n+1),\,1]$,
 which together with \cite{t83} implies that
$\dot M^{1,\,p}(\rn)=\dot M^{1,\,p}_\ball(\rn)=\dot F^1_{p,\,2}(\rn)$ for all $p\in(n/(n+1),\,\fz)$,
while for all $s\in(0,\,1)$ and $p\in(n/(n+s),\,\fz)$,
 $\dot M^{s,\,p}(\rn)=\dot M^{s,\,p}_\ball(\rn)=\dot F^s_{p,\,\fz}(\rn)$
as proved in \cite{y03,kyz09}.

Now we recall some notions  on imbeddings.
Let $\boz$ be a bounded domain of $\rn$, $s\in(0,\,1]$ and $p\in(n/(n+s),n/s)$.
Then $\boz$ is said to support a $(pn/(n-ps),\,p)_s$-Haj\l asz-Sobolev-Poincar\'e
(for short, $(pn/(n-ps),\,p)_s$-{HSP}) imbedding if there exists a  constant $C>0$ such that
for all $u\in\dot M^{s,\,p}_\ball(\boz)$,
\begin{equation}
 \label{e1.2}
\|u-u_\boz\|_{L^{pn/(n-ps)}(\boz)}\le C\|u\|_{\dot M^{s,\,p}_\ball(\boz)},
\end{equation}
where $u_\boz\equiv\frac1{|\boz|}\int_\boz u(z)\,dz$.
Similarly, $\boz$ is said to support an $s$-Haj\l asz-Trudinger (for short, $s$-HT)
imbedding if there exists a  constant $C>0$ such that for all $u\in\dot M^{s,\,n/s}_\ball(\boz)$,
\begin{equation}\label{e1.3}
 \|u-u_\boz\|_{\phi_s(L)(\boz)}\le C\|u\|_{\dot M^{s,\,n/s}_\ball(\boz)},
\end{equation}
where and in what follows, $\phi_s(t)\equiv\exp(t^{n/(n-s)})-1$ and
\begin{equation} \label{e1.4}
\|u\|_{\phi_s(L)(\boz)}\equiv\inf\lf\{t>0,\,\int_\boz\phi_s\lf(\frac{|u(x)|}t\r)\,dx\le1\r\}.
\end{equation}
It should be pointed out that since $\dot M^{1,\,p}_\ball(\boz)=\dot W^{1,\,p}(\boz)$ for all $p\in(1,\,\fz)$,
 then \eqref{e1.2} with $s=1$ and $p\in[1,\,n)$ coincides with the classical $(pn/(n-p),\,p)$-Sobolev-Poincar\'e imbedding as in \cite[(1.1)]{bk96},
and \eqref{e1.3} with $s=1$ coincides with the classical Trudinger imbedding as in \cite[(1.2)]{bk96}.

Recently, some geometric criteria were established in \cite{b88,bk95,bk96} for a domain  to support
a $(pn/(n-p),\,p)$-Sobolev-Poincar\'e imbedding for $p\in[1,\,n)$ or a Trudinger imbedding.
More precisely, Bojarski \cite{b88} first proved that a John domain as in Definition \ref{d2.1}
 always supports a $(pn/(n-p),\,p)$-Sobolev-Poincar\'e  imbedding for all $p\in[1,\,n)$.
Smith and Stegenga \cite{ss90} proved that a weak carrot domain as in Definition \ref{d2.2}
always supports the Trudinger imbedding.
Conversely, let $\boz$ be a bounded planar domain or a bounded domain in $\rn$ with $n\ge 3$
satisfying an additional separation property when $p\in(1,\,n)$
and a slice property when $p=n$; see Definitions \ref{d2.3} and \ref{d2.4} below.
Then Buckley and Koskela \cite{bk95,bk96} proved that
if $\boz$  supports a $(pn/(n-p),\,p)$-Sobolev-Poincar\'e imbedding
for some/all $p\in[1,\,n)$, then it is a John domain,
and if $\boz$ supports the Trudinger imbedding, then it is a weak carrot domain.

The purpose of this paper is to
establish  some geometric criteria for a domain of $\rn$ with $n\ge2$
to support a $(pn/(n-ps),\,p)_s$-HSP  imbedding with $s\in(0,\,1]$ and $p\in(n/(n+s),\,n/s)$
or an $s$-HT imbedding with $s\in(0,\,1]$.

To this end, we first establish the linear local connectivity (for short, LLC) of a domain that supports
the $(pn/(n-ps),\,p)_s$-{HSP} imbedding,
where the notion of LLC was introduced by Gehring \cite{g77}.
Recall that a domain $\boz$ is said to have the LLC property if
there exists a positive constant $b$ such that for all $z\in\rn$ and $r>0$,

$\llc(1)$ \quad points in $\boz\cap B(z,\,r)$ can be joined in $\boz\cap B(z,\,r/b)$;

$\llc(2)$ \quad points in $\boz\setminus B(z,\,r)$ can be joined in $\boz\setminus B(z,\,br)$.

\noindent Then, as proved by Gehring and Martio \cite{gm85b},   a $\dot W^{1,\,n}$-extension domain has
the LLC property, and by
 \cite[Theorem  6.4]{k90},
a $\dot W^{1,\,p}$-extension domain  with $p\in(n-1,\,n)$ has the LLC(2) property; see also \cite{g82,gr83,gv,vgl}
and their references.
Here and in what follows,
$\boz$ is called an $A$-extension domain
with $A=\dot M^{s,\,p}_\ball$, $\dot W^{1,\,p}$ or $\dot H^{1,\,p}$
if for every $u\in  A(\boz)$, there exists a $v\in   A(\rn)$
such that $v|_\boz=u$ and $\|v\|_{A(\rn)}\ls
\|u\|_{A(\boz)} $.
Here, we extend the results in \cite{gm85b,k90} as follows. 

\begin{thm}\label{t1.1} Let  $s\in(0,\,1]$ and $p\in(n/(n+s),\,n/s)$. 
If $\boz$ is a bounded $\dot M^{s,\,p}_\ball$-extension domain
or $\boz$ is a bounded domain that supports a $(pn/(n-ps),\,p)_s$-{\rm HSP} imbedding, 
then $\boz$ has the {\rm LLC(2)} property.
\end{thm}

The proof of Theorem \ref{t1.1} is given in Section \ref{s3}.
We point out that the approach used here is different from that used by Koskela in
\cite[Theorem 6.4]{k90}, where he  used the $p$-capacity to prove the LLC(2)
property of a $\dot W^{1,\,p}$-extension domain for $p\in(n-1,\,n)$.
In fact, when $1<p\le n-1$, as Koskela \cite{k90} pointed out,
the $p$-capacity makes no sense
since ${\rm Cap}_p(K_0,\,K_1,\,\rn)=0$ for every pair of disjoint continua
$K_0,\,K_1\subset\rn$.
So some new ideas are required to prove Theorem \ref{t1.1} 
as the result is new even in the case $s=1$ and $1<p\le n-1$. 
To this end, we will simplify this question,
 and then combine  some of the ideas from \cite{bk95,hkt08,hkt08b}
and the properties of  Haj\l asz-Sobolev functions.

Then, as a corollary to Theorem \ref{t1.1}, we have the following conclusion,
which complements the results in \cite{gm85b,k90}.

\begin{cor}\label{c1.1}
If $\boz$  is a bounded $\dot W^{1,\,p}$-extension domain when $p\in(1,\,n)$
or bounded $\dot H^{1,\,p}$-extension domain with $p\in(n/(n+1),\,1]$,
then $\boz$ has the {\rm LLC(2)} property.
\end{cor}

Applying Theorem \ref{t1.1}, we further establish some geometric criteria for a domain to support  a
$(pn/(n-ps),\,p)_s$-HSP imbedding, which generalizes the criteria in \cite{b88,bk95}.

\begin{thm}\label{t1.2}
(i) A John domain of $\rn$ as in Definition \ref{d2.1} always supports a
 $(pn/(n-ps),\,p)_s$-{\rm HSP} imbedding as in \eqref{e1.2} for all $s\in(0,\,1]$
and  $p\in(n/(n+s),\,n/s)$.

(ii) Assume that $\boz$ is a bounded domain of $\rn$ and satisfies the separation property as in Definition \ref{d2.3}. 
If $\boz$ supports  a $(pn/(n-ps),\,p)_s$-{\rm HSP} imbedding for some  $s\in(0,\,1]$
and  $p\in(n/(n+s),\,n/s)$, then $\boz$ is a John domain.
\end{thm}

To prove Theorem \ref{t1.2}(ii), we will use the LLC(2) property of these domains given in
Theorem \ref{t1.1}. 
This is slightly different from that of \cite{bk95}.
On the other hand, notice that $(\rn,\,d_s,\,dx)$ is an Ahlfors $n/s$-regular
metric measure spaces, when   $d_s(x,\,y)=|x-y|^s$ for all $x,\,y\in\rn$ and
$dx$ denotes the Lebesgue measure.
 Observe that $M^{s,\,n/s}_\ball(\boz)$ coincides with $\dot M^{1,\,n/s}_\ball(\boz,\,d_s,\,dx)$,
the Haj\l asz-Sobolev space 
on domains of $(\rn,\,d_s,\,dx)$ defined similarly to Definition \ref{d1.1}.
Then Theorem \ref{t1.2}(i) can be dudeced from results by Chua and Wheeden \cite{cw}.
For the reader's convenience, we give a short proof, which will
use the ideas from Bojarski \cite{b88},
the chain property of a John domain as proved by Boman \cite{b},
and a key imbedding on balls established by Haj\l asz \cite[Theorem 8.7]{h03}.

We also establish an analogue of Theorem \ref{t1.2} at the end point $p=n/s$ when $s\in(0,\,1]$,
which generalizes the criteria established in \cite{ss90,bk96,bo99},
and whose proof uses some ideas from \cite{ss90,ss90b,bk96,bo99} and will be given in Section \ref{s4}.
Also see \cite{mp} for similar inequalities on balls.

\begin{thm}\label{t1.3}
(i) A weak carrot domain of $\rn$ as in Definition \ref{d2.2} always supports
an $s$-{\rm HT} imbedding for all $s\in(0,\,1]$.

(ii) Assume that $\boz$ is a bounded domain of $\rn$ and  satisfies the slice property as in Definition \ref{d2.4}.
If $\boz$ supports an $s$-{\rm HT} imbedding for some  $s\in(0,\,1]$,
then $\boz$ is a weak carrot domain.
\end{thm}

Notice that, as proved in \cite{bk95,bk96}, every simply connected domain in $\rr^2$ or
every domain in $\rr^n$ with $n \ge 3$ that is quasiconformally equivalent to a uniform domain satisfies the slice property
and the separation property. So, as a corollary to Theorems \ref{t1.2} and \ref{t1.3},
we have the following conclusion.

\begin{cor}\label{c1.2} Let $\boz$ be a bounded simply connected domain  in $\rr^2$ or  a bounded
domain in $\rn$ with $n\ge3$ that is quasiconformally equivalent to a uniform domain.
Then

(i) $\boz$ is a John domain if and only if it supports a $(pn/(n-ps),\,p)_s$-{\rm HSP} imbedding for some/all  $s\in(0,\,1]$
and  $p\in(n/(n+s),\,n/s)$;

(ii) $\boz$ is a weak carrot domain if and only if it supports an $s$-{\rm HT} imbedding for some/all $s\in(0,\,1]$.
\end{cor}

This paper is organized as follows. In Section \ref{s2}, we recall some basic notions and properties  of the
domains and Haj\l asz-Sobolev spaces. In Section \ref{s3}, we present the proof of Theorem \ref{t1.1}.
In Section \ref{s4}, we give the proofs of Theorems \ref{t1.2} and \ref{t1.3}.

\section{Preliminaries}\label{s2}

In this section, we recall some  notions and basic properties of
domains and  Haj\l asz-Sobolev spaces.
We begin with the notion of John domain.

\begin{defn}\label{d2.1}  \rm
Let $\boz$ be a bounded domain of $\rn$ with $n\ge2$. Then
$\boz$ is called a John domain with respect to
  $x_0\in\boz$ and  $C>0$  if
 for every $x\in\boz$,
there exists a rectifiable curve $\gz:[0,\,1]\to\boz$ parametrized by arclength such that
$\gz(0)=x$, $\gz(1)=x_0$ and $d(\gz(t),\,\boz^\complement)\ge Ct$.
\end{defn}

Now we recall the notion of a weak carrot domain (or domains satisfying the quasihyperbolic boundary condition).
To this end, for every pair of points $x,\,y\in\boz$, define their quasihyperbolic distance $k_\boz(x,\,y)$  by
$$k_\boz(x,\,y)\equiv \inf_\gz\int_{\gz}\frac1{d(z,\,\boz^\complement)}\,|dz|,$$
where the infimum is taken over all rectifiable curves $\gz\subset\boz$ joining $x$ and $y$.
As proved in \cite{go79}, $k_\boz$ is a geodesic distance, namely,
there exists a curve $\gz_{x,\,y}\subset\boz$ such that  $$k_\boz(x,\,y)=\int_{\gz_{x,\,y}}\frac1{d(z,\,\boz^\complement)}\,|dz|.$$

\begin{defn}\label{d2.2}  \rm
A domain $\boz$ is said to satisfy a weak carrot condition (or
quasihyperbolic boundary condition)
 with respect to $x_0\in\boz$ and $C\ge1$ if for all $x\in\boz$,
\begin{equation}\label{e2.1}
 k_\boz(x,\,x_0)\le C\log\lf(\frac C{d(x,\,\boz^\complement)}\r).
\end{equation}

\end{defn}

It is easy to see that the John and weak carrot conditions are independent of the choice of $x_0$ in the sense that
if $\boz$ is a John or  weak carrot domain with respect to   $x_0$ and  $C$,
then for any other $x_1\in\boz$, there exists a positive constant $\wz C$ such that
$\boz$ is still a John or weak carrot domain with respect to  $x_1$ and $\wz C$, respectively.
See \cite{bo99} for more details.

The following characterization of a weak carrot domain established by
 Smith and  Stegenga \cite{ss90} will be used in the proof of Theorem \ref{t1.3}.

\begin{lem}\label{l2.1} Let $\boz$ is a proper subdomain of $\rn$ and let $x_0\in\boz$.
Then $\boz$ is a weak carrot domain if and only if there exists a positive constant $\sz$
such that $$\int_\boz\exp(\sz k_\boz(x_0,\,x)))\,dx<\fz.$$
\end{lem}

We also recall the notions of a separation property and slice property introduced in \cite{bk95,bk96}.

\begin{defn}\label{d2.3}
A domain $\boz$ has a separation property with respect to  $x_0\in\boz$ and $C>1$
if for every $x\in\boz$, there exists a curve $\gz:\,[0,\,1]\to \boz$ with $\gz(0)=x$,
$\gz(1)=x_0$, and such that for each $t\in(0,\,1]$, either
$\gz([0,\,t])\subset B\equiv B(\gz(t),\,Cd(\gz(t),\boz^\complement))$
or each $y\in\gz([0,\,t])\setminus B$ belongs to a different
component of $\boz\setminus \partial B$ than $x_0$.
\end{defn}

\begin{defn}\label{d2.4}
A domain $\boz$ has a slice property with respect to  $C>1$
if for every pair of points $x,\,y \in\boz$, there exists  a rectifiable curve $\gz:[0,\,1]\to\boz$ with
$\gz(0)=x$ and  $\gz(1)=y$, and pairwise disjoint collection of open subsets
$\{S_i\}_{i=0}^j$, $j\ge0$, of $\boz$ such that

(i) $x \in S_0$, $y\in S_j$ and $x $ and $y$ are in different components of $\boz\setminus \overline{S_i}$ for $0<i<j$;

(ii) if $F\subset\subset  \boz$ is a curve containing both $x$ and $y$, and $0<i<j$, then $\diam(S_i)\le C\ell(F\cap S_i)$;

(iii) for $0\le t\le 1$, $B(\gz(t),\,C^{-1}d(\gz(t),\,\boz^\complement))\subset\cup_{i=0}^jS_i$;

(iv) if $0\le i\le j$, then $\diam S_i\le C d(z,\,\boz^\complement)$ for all $z\in\gz_i\equiv\gz\cap S_i$;
also, there exists $x_i\in S_i$ such that $x_0=x$, $x_j=y$ and
$B(x_i,\,C^{-1}d(x_i,\,\boz^\complement))\subset S_i$.
\end{defn}

We point out that, as proved in \cite{bk95,bk96}, every simply connected domain in $\rr^2$ or
every domain in $\rr^n$ with $n \ge 3$ that is quasiconformally equivalent to a uniform domain satisfies a slice property
and a separation property. Every John domain satisfies both a separation and a slice property; see \cite{bo99}.

The following conclusion is essentially established in \cite{ks08} and
plays an important role in the proofs of Theorems \ref{t1.1}, \ref{t1.2}  and \ref{t1.3}.
For every $\rho>0$, similarly to $\cd_\ball^s(u)$, we denote by $\cd_\ball^{s,\,\rho}(u)$ the collection of
all measurable functions $g$ such that
 \eqref{e1.1} holds for all $x,\ y\in\boz\setminus E$
satisfying $|x-y|<\rho\dist(x,\,\partial\boz)$.
Notice that $\cd_\ball^s(u)= \cd_\ball^{s,\,1/2}(u)$ and
$\cd^s(u)= \cd_\ball^{s,\,\fz}(u)$.

\begin{lem}\label{l2.3}
Let $s\in(0,\,1]$ and $p\in(n/(n+s),\,\fz)$. Then $u\in\dot M^{s,\,p}_\ball(\boz)$ if and only if
there exists a $\rho\in(0,\,1)$ such that
$\inf_{g\in\cd^{s,\,\rho}_\ball(u)}\|g\|_{L^p(\boz)}<\fz$.  
Moreover, for given $\rho$,  there exists a positive constant $C$ such that
for all $u\in\dot M^{s,\,p}_\ball(\boz)$, 
$$C^{-1}\|u\|_{\dot M^{s,\,p}_\ball(\boz)}\le \inf_{g\in\cd^{s,\,\rho}_\ball(u)}\|g\|_{L^p(\boz)}\le C\|u\|_{\dot M^{s,\,p}_\ball(\boz)}.$$
\end{lem}

We also need the following imbedding,
which is essentially established by Haj\l asz \cite[Theorem 8.7]{h03} when $n=1$ and
pointed out by Yang \cite{y03} when $s\in(0,\,1)$.

\begin{lem}\label{l2.4}
Let $s\in(0,\,1]$ and $p\in(n/(n+s),\,n/s)$.
Then for every $\sz>1$, there exists a positive $C$ constant such that for all balls or cubes $B$ and
$u\in\dot M^{s,\,p}(\sz B)$,
$$
 \|u -u_B\|_{L^{pn/(n-ps)}(B)}\le C\|u\|_{\dot M^{s,\,p}(\sz B)}.
$$
\end{lem}

By Lemma \ref{l2.4}, we have the following conclusion.

\begin{lem}\label{l2.5}
 Let $s\in(0,\,1]$ and $p\in(n/(n+s),\,n/s)$.
Then a bounded  $\dot M^{s,\,p}_\ball$-extension domain always
supports a $(pn/(n-ps),\,p)_s$-{\rm HSP} imbedding.
\end{lem}

\begin{proof}
Assume that
$\boz$ is an $\dot M^{s,\,p}_\ball$-extension domain.
Let $u\in\dot M^{s,\,p}_\ball(\boz)$.
Then there exists a $v\in\dot M^{s,\,p}_\ball(\rn)$ such that $v|_\boz=u$ and
$\|v\|_{\dot M^{s,\,p}_\ball(\rn)}\ls \|u\|_{\dot M^{s,\,p}_\ball(\boz)}$.
Let $B$ be a ball of $\rn$ such that $\boz\subset B $.
Then $v\in \dot M^{s,\,p}_\ball(2B)$ and thus by Lemma \ref{l2.3}, we have
$v\in L^{pn/(n-ps)}(B)$  and
$$\|v-v_B\|_{L^{pn/(n-ps)}(B)}\ls\|v\|_{\dot M^{s,\,p}_\ball(\rn)}\ls \|v\|_{\dot M^{s,\,p}_\ball(\boz)}$$
which further implies that
$$\|u-u_\boz\|_{L^{pn/(n-ps)}(\boz)}\le \|v-v_\boz\|_{L^{pn/(n-ps)}(B)}\ls\|v-v_B\|_{L^{pn/(n-ps)}(B)}\ls
  \|v\|_{\dot M^{s,\,p}_\ball(\boz)}.$$
This means that $\boz$  supports a $(pn/(n-ps),\,p)_s$-HSP imbedding
and thus finishes the proof of Lemma \ref{l2.5}.
\end{proof}

Finally, we state some conventions. Throughout the paper,
we denote by $C$ a positive constant which is independent
of the main parameters, but which may vary from line to line.
Constants with subscripts, such as $C_0$, do not change
in different occurrences. The symbol $A\ls B$ or $B\gs A$
means that $A\le CB$. If $A\ls B$ and $B\ls A$, we then
write $A\sim B$.
For any locally integrable function $f$,
we denote by $\bbint_E f  $ the average
of $f$ on $E$, namely, $\bbint_E f  \equiv\frac 1{|E|}\int_E f\,dx$.

\section{Proof of Theorem \ref{t1.1}}\label{s3}

\begin{proof}[Proof of Theorem \ref{t1.1}]
By Lemma \ref{l2.5},
it  suffices to prove that a domain which
supports a  $(pn/(n-ps),\,p)_s$-HSP imbedding
has the LLC(2) property.
Assume that $\boz$ is a bounded domain that supports a
$(pn/(n-ps),\,p)_s$-HSP imbedding.
We want to show that $\boz$ has the LLC(2) property.
To this end, let $L\equiv \diam\boz$  and $x_0\in \boz$ be such that
$r_0\equiv d(x_0,\,\boz^\complement)=\max\{d(x,\boz^\complement):\ x\in\boz\}$.
Notice that if $u(y)=0$ for all $y\in B(x_0,\,r_0)$,
then the $(pn/(n-ps),\,p)_s$-HSP imbedding implies that
\begin{equation}\label{e3.1}
 \|u\|_{L^{pn/(n-ps)}(\boz)}\ls \|u\|_{\dot M^{s,\,p}_\ball( \boz)},
\end{equation}
where the constant depends on $r_0$ and $|\boz|$ but not on $u$.

We  claim that
if $x,\,x_0\in\boz\setminus B(z,\,r)$ for $z\in B(x_0,\,2L)$ and $r\in(0,\,2L)$,
then $x,\,x_0$ are contained in the same component of
$\boz\setminus B(z,\,br)$ for some fixed constant $b\in(0,\,1)$,
which may depend on $\boz$ and $x_0$ but not on $z$ and $x$.

Assume that the above claim holds for the moment.
Then we deduce Theorem \ref{t1.1} from it by the following 2 steps.
Let $x,\,y\in\boz\setminus B(z,\,r)$ for $z\in\rn$ and $r\in(0,\,\fz)$.

{\it Step 1.}  There exists a positive constant $\wz b$ independent of $x$ such that
if $x,\,x_0\in\boz\setminus B(z,\,r)$, then $x,\,x_0$ are contained in the same component of
$\boz\setminus B(z,\,\wz br)$.
To see this, assume that $x_0\in\boz\setminus B(z,\,r)$.
If $z\notin B(x_0,\,2L)$,
then $\boz\cap B(z,\,r) \ne \emptyset$ implies that
$r\ge d(z,\,x_0)-L\ge r/2\ge L$, and moreover $\boz\setminus B(z,\,r)\ne\emptyset$ implies that
$\boz\cap B(z,\,r/2)=\emptyset$.
Thus if $x_0,\,x\in\boz\setminus B(z,\,r)$ with $d(z,\,x_0)\ge 2L$,
then $x_0,\,x$ are contained in the same component of
$\boz\setminus B(z,\,r)$ if $\boz\setminus B(z,\,r)=\emptyset$
or of $\boz\setminus B(z,\,r/2)$ if $\boz\setminus B(z,\,r)\ne\emptyset$.
If $z\in B(x_0,\,2L)$, then by the above claim, it suffices to consider the case $r\ge 2L$.
Since $r\ge 2L$ implies $d(z,\,\boz)\ge r-L\ge r/2$, which means that $\boz\cap B(z,\,r/2)=\emptyset$,
then $x_0,\,x$ are contained in the same component of $\boz\setminus B(z,\,r/2)$.

{\it Step 2.}  There exists a positive constant $b$ independent of $x,\,y$ such that
$x,\,y$ are contained in the same component of
$\boz\setminus B(z,\,br)$.
To see this, if $x_0\in \boz\setminus B(z,\,\frac{r_0}{10L}r)$, then
$x_0, x$ and $x_0,\,y$, and thus  $x,\,y$,  are contained in the same component of
$\boz\setminus B(z,\,\wz b\frac{r_0}{10L}r)$.
If $x_0\in B(z,\,\frac{r_0}{10L}r)$, then $r-\frac{r_0}{10L}r\le L$, which implies that
$r\le 2L$ and thus  $|z-x_0|\le \frac{r_0}{10L}r\le r_0/5$.
Obviously, $$B(z,\, \frac{r_0}{10L}r)\subset B(x_0,\, \frac{r_0}{5L}r)\subset B(x_0,\, r_0)\subset\boz,$$
which means that $\boz\setminus B(z,\, \frac{r_0}{10L}r)$ is connected,
and thus $x,\,y$ are contained in the same component of $ \boz \setminus B(z,\,\frac{r_0}{10L}r)$.

Therefore,  with the aid of  the above claim,
combining Step 1 and Step 2, we obtain Theorem \ref{t1.1}.
So we have  reduced Theorem \ref{t1.1} to the above claim.
The remainder of the proof of Theorem \ref{t1.1} consists of the proof of the above claim.

In the following argument, we let  $x\in\boz$, $z\in B(x_0,\,2L)$ and
 $r\in(0,\,2L)$ be fixed such that
 $x,\,x_0\in\boz\setminus B(z,\,r)$ as in the claim.
Let $b_z\in(0,\,1]$ be the supremum of $b\in(0,\,1)$
such that $x,\,x_0$ are contained in the same component of $\boz\setminus \overline{B(z,\,br)}$.
Without loss of generality, we assume that $b_z\le 1/10$.
Denote by $\boz_x$ the component of $\boz\setminus \overline{B(z,\,b_0r)}$ with $b_0=2b_z$  containing $x$.
Take $b_1\in(b_0,\,1]$ such that
$$
|\boz_x\cap (B(z,\,r)\setminus B(z,\,b_1r))|=\frac12 |\boz_x\cap (B(z,\,r)\setminus B(z,\,b_0r))|=
\frac12 |\boz_x\cap B(z,\,r)|.
$$
Define a function $u$ on $\boz$ by setting
\begin{equation}\label{e3.2}
u(y)\equiv\left\{\begin{array}{ll}
0, &  y\in\boz\setminus\boz_x;\\
\dfrac{d(y,\,B(z,\,b_0r))}{b_1r-b_0r}, \quad & y\in \boz_x\cap B(z,\,b_1r);\\
1, &y\in \boz_x\setminus B(z,\,b_1r).
\end{array}\right.
\end{equation}
Then we have the following conclusion, whose proof will be given below.

\begin{lem}\label{l3.1}
Let $u$ be as in \eqref{e3.2} and
$s\in(0,\,1]$. 
 Then $g\equiv C(b_1r-b_0r)^{-s}\chi_{\boz_x\cap B(z,\,r)}$
is  an element of $\cd^{s,\,1/8}_\ball(u)$, where $C$ is a positive constant independent of $u,\,x,\,b_0,\,b_1,\,r.$.
\end{lem}

By Lemma \ref{l2.3}, Lemma \ref{l3.1} and \eqref{e3.1}, we further have $u\in \dot M^{s,\,p}_\ball( \boz)$ and
$$\|u\|_{L^{pn/(n-ps)}(\boz)}\ls \|u\|_{\dot M^{s,\,p}_\ball( \boz)}\ls
\|g\|_{L^{p}(\boz)}\ls (b_1r-b_0r)^{-s}|\boz_x\cap (B(z,\,r)\setminus B(z,\,b_0r))|^{1/p},$$
which together with
\begin{eqnarray*}
\|u\|_{L^{pn/(n-ps)}(\boz)}&&\gs |\boz_x\cap (B(z,\,r)\setminus B(z,\,b_1r))|^{(n-ps)/pn}\\
&& \gs
|\boz_x\cap (B(z,\,r)\setminus B(z,\,b_0r))|^{(n-ps)/pn}
\end{eqnarray*}
implies that
\begin{equation}\label{e3.3}
b_1r-b_0r\ls |\boz_x\cap (B(z,\,r)\setminus B(z,\,b_0r))|^{1/n}.
\end{equation}
Hence, if $b_1\ge 1/2$, then  \eqref{e3.3} implies that
\begin{equation}\label{e3.4}
(1/2-b_0)r \ls |\boz_x\cap (B(z,\,r)\setminus B(z,\,b_0r))|^{1/n}.
\end{equation}
If $b_1<1/2$, then following the above procedure, we can find a sequence $\{b_j\}_{j=1}^{j_0}$ such that
$b_{j_0}\ge 1/2$ and for all $0\le j\le j_0-1$, $b_j<1/2$,
$$|\boz_x\cap (B(z,\,r)\setminus B(z,\,b_{j+1}r))|
=\frac12|\boz_x\cap (B(z,\,r)\setminus B(z,\,b_jr))|,
$$
and
$$b_{j+1}r-b_jr\ls |\boz_x\cap (B(z,\,r)\setminus B(z,\,b_jr))|^{1/n}.$$
This implies that
$$ \sum_{j=0}^{j_0-1}(b_{j+1}r-b_jr)
\ls |\boz_x\cap (B(z,\,r)\setminus B(z,\,b_0r))|^{1/n},
$$
and hence   \eqref{e3.4}.
To control $|\boz_x\cap (B(z,\,r)\setminus B(z,\,b_0r))|^{1/n}$ via $b_0r$, define function
\begin{equation}\label{e3.5}
v(y)\equiv\inf_{\gz(x_0,\,y)}\ell(\gz\cap B(z,\,b_0r))
\end{equation}
for all $y\in\boz$, where the infimum is taken over all the rectifiable curves $\gz$ joining $x_0$ and $y$ in $\boz$.
Observe that for all $y$ in the component of $\boz\setminus \overline{B(z,\,b_0r)}$ containing $x_0$,
$v(y)=0$;  for all $y$ in the component $\boz_x\setminus \overline{B(z,\,b_0r)}$
which contains $x$ and does not contain $x_0$,
$v(y)$ is a constant larger than or equal to $b_0r$.
Moreover,  we have the following conclusion, whose proof will be given below.

\begin{lem}\label{l3.2}
Let $v$ be as in \eqref{e3.5} and $s\in(0,\,1]$.
 Then
$h\equiv C(b_0r)^{1-s}\chi_{\boz\cap B(z,\,b_0r)}$ is   an element of $\cd^{s,\,1/8}_\ball(v)$,
where $C$ is a positive constant independent of $v,\,z,\,b_0,\,r$.
\end{lem}

By Lemma \ref{l2.3}, Lemma \ref{l3.2} and \eqref{e3.1}, we have that $v\in \dot M^{s,\,p}_\ball(\boz)$ and
$$(b_0r)|\boz_x\cap  (B(z,\,r)\setminus B(z,\,b_0r) )|^{(n-ps)/pn}\ls|\boz\cap B(z,\,b_0r)|^{1/p}(b_0r)^{1-s}$$
which implies that $$ |\boz_x\cap   (B(z,\,r)\setminus B(z,\,b_0r) ) | \ls (b_0r)^n.$$
By this and \eqref{e3.4}, we have
$(1/2-b_0)r\ls b_0r$, which implies that
$b_0\ge C$ for some fixed constant $C\in(0,\,1)$ independent of $x$.
This gives the above claim by taking $b=C/4$
and thus finishes the proof of Theorem \ref{t1.1}.
\end{proof}

\begin{proof}[Proof of Lemma \ref{l3.1}]
It suffices to check that for every pair of
 $y,\,w\in\boz$ such that $|y-w|<\dist(y,\,\boz^\complement)/8$,
\begin{equation}\label{e3.6}
 |u(y)-u(w) |\ls \frac{|y-w|^s}{(b_1r-b_0r)^{s}}[\chi_{\boz_x\cap B(z,\,r)}(y)+ \chi_{\boz_x\cap B(z,\,r)}(w)].
\end{equation}
To prove \eqref{e3.6}, without loss of generality, we may assume that $u(w)<u(y)$.
Then $u(y)>0$ implies that $y\in \boz_x$ and $u(w)<1$ implies that $w\notin\boz_x\setminus B(z,\,b_1r)$.
We will consider the following three cases for $w$:
i) $w\in \boz_x\cap  B(z,\,b_1r)$; ii) $w\in \boz\cap \overline{ B(z,\,b_0r)}$;
iii) $w\in\boz\setminus(\boz_x\cup \overline {B(z,\,b_0r))}$.

{\it Case i).}
If $y\in \boz_x\setminus B(z,\,b_1r)$, then by $w\in \boz_x\cap  B(z,\,b_1r)$,
we have
$$d(w,\,B(z,\,b_0r))=|w-z|-b_0r\ge |z-y|-|w-y|-b_0r\ge (b_1r-b_0r)-|w,\,y|,$$
and thus  $$|u(y)-u(w)|=\lf|1-\frac{d(w,\,B(z,\,b_0r))}{b_1r-b_0r}\r|\le
\lf|1-\frac{d(w,\,B(z,\,b_0r))}{b_1r-b_0r}\r|^s\le
\frac{|w-y|^s}{(b_1r-b_0r)^s},$$
 which gives \eqref{e3.6}.
If $y\in\boz_x\cap B(z,\,b_1r) $, then by $|w-y|\le b_1r-b-0r $,
$$|u(y)-u(w)|=\lf| \frac{d(y,\,B(z,\,b_0r))-d(w,\,B(z,\,b_0r))}{b_1r-b_0r}\r|\le \frac{|w-y|}{b_1r-b_0r}
\le\frac{|w-y|^s}{(b_1r-b_0r)^s},$$
 which gives \eqref{e3.6}.

{\it Case ii).} If $y\in \boz_x\cap B(z,\,r)$,
then $d(y,\,B(z,\,b_0r))\le |y-w|\le b_1r-b_0r$ and thus
$$|u(y)-u(w)|=\lf|\min\lf\{1,\,\frac{d(y,\,B(z,\,b_0r))}{b_1r-b_0r}\r\}\r|\le \lf|\min\lf\{1,\,\frac{|w-y|}{b_1r-b_0r}\r\}\r|^s
\le\frac{|w-y|^s}{(b_1r-b_0r)^s},$$
 which gives \eqref{e3.6}.
If $y\in \boz_x\setminus B(z,\,r)$, then $|w-y|\ge (1-b_0)r\ge b_0r$.
Since $d(w,\,\boz^\complement)\le |w-z|+d(z,\,\boz^\complement)\le 2b_0r$,
so we have that  $|w-y|\ge 2 d(w,\,\boz^\complement)$.
Moreover, since $d(y,\,\boz^\complement)\le |y-w|+d(w,\,\boz^\complement)\le 2|y-w|$,
 by the definition of $\cd^{s,\,1/8}_\ball(u)$,
we do not need to check \eqref{e3.6} for $y\in \boz_x\setminus B(z,\,r)$.

{\it Case iii).}
 We will prove that in this case,
  \begin{equation}\label{e3.7}
|y-w|\ge\frac18\max\{ d(y,\,\boz^\complement),\,  d(w,\,\boz^\complement)\}.
\end{equation}
 Thus, we do not need to check \eqref{e3.6} by the definition of $\cd^{s,\,1/8}_\ball (u)$.
To prove \eqref{e3.7}, notice that $y\in\boz_x$ and $w\notin\boz_x\cup \overline {B(z,\,b_0r)}$ implies that
 $y$ and $w$ are in different components of $\boz\setminus  \overline {B(z,\,b_0r)}$.
If $|y-w|< d(y,\,\boz^\complement)/4$,
then  $B(y,\,2|w-y|)\subset B(y,\,d(y,\,\boz^\complement))\subset\boz$.
Observe that either $B(y,\,2|w-y|)\setminus \overline{B(z,\,b_0r)}=\emptyset$  or
$B(y,\,2|w-y|)\setminus \overline{B(z,\,b_0r)}$ is connected.
So  $y,\,w \in  B(y,\,2|w-y|)\setminus \overline {B(z,\,b_0r)}\subset\boz$ means that
$B(y,\,2|w-y|)\setminus \overline {B(z,\,b_0r)}$ are connected and thus
$y,\,w $ are in the same component of $\boz\setminus \overline {B(z,\,b_0r)}$,
which is a contradiction.
If $|y-w|< d(y,\,\boz^\complement)/8$, then for all $z\ne\boz$,
$$d(z,\,y)\ge d(z,\,w)-d(w,\,y)\ge d(w,\,\boz^\complement)/2,$$  which implies that
$d(y,\,\boz^\complement)\ge 7d( w,\,\boz^\complement)/8$ and thus
$|y-w|< d(y,\,\boz^\complement)/4$.
Thus, \eqref{e3.7} holds. This finishes the proof of Lemma \ref{l3.1}.
\end{proof}

\begin{proof}[Proof of Lemma \ref{l3.2}]
This is quite similar to the proof of Lemma \ref{l3.1}.
We sketch the proof.
It suffices to check that for every pair of
 $y,\,w\in\boz$ such that $|y-w|<\dist(y,\,\boz^\complement)/8$,
\begin{equation}\label{e3.8}
 |v(y)-v(w) |\ls (b_0r)^{1-s} [\chi_{\boz_x\cap B(z,\,b_0r)}(y)+ \chi_{\boz_x\cap B(z,\,b_0r)}(w)].
\end{equation}
If both $y$ and $w$  are in the same component of $\boz\setminus \overline{B(z,\,b_0r)}$, then \eqref{e3.8} holds.
If $y,\,w$ are in different components of $\boz\setminus\overline{ B(z,\,b_0, r)}$,
by an argument similar to that of \eqref{e3.7}, we can prove that
\eqref{e3.7} still holds, and thus we do not need to check \eqref{e3.8} for such $y,\,w$.
So we can assume that one of  $w,\,y$ is in $\boz\cap\overline{ B(z,\,b_0r)}$.
Notice that in this case $I(w,\,y)\subset\boz$, where $I (w,\,y)$ denotes the line segment joining $w$ and $y$.
Then
$$
 |u(y)-u(w) |\le \ell(I(w,\,y)\cap B(z,\,b_0r))\le \min\{2b_0r,\,|y-w|\}\ls (b_0r)^{1-s}|y-w|^s,
$$
which implies \eqref{e3.8}.
This finishes the proof of
Lemma \ref{l3.2}.
\end{proof}

\section{Proofs  of Theorems \ref{t1.2} and \ref{t1.3}}\label{s4}

\begin{proof} [Proof of Theorem \ref{t1.2}]
(i) Assume that $\boz$ is a bounded John domain.
Then, as proved by Boman \cite{b},
$\boz$ enjoys the following chain property:
there exist a positive constant $ \wz C$ and
a sequence of subcubes of $\boz$, which is denoted by $\cf$, such that

(a) $\chi_\boz(x)\le\sum_{j} \chi_{Q_i}(x)\le\sum_{j} \chi_{2Q_i}(x)\le \wz C\chi_\boz(x)$ for all $x\in\rn$;

(b) for a fixed subcube $Q_0\in\cf$ and any other $Q\in\cf$,
there exists a subsequence $\{Q_j\}_{j=1}^N\subset\cf$ satisfying that
$Q=Q_N\subset \wz CQ_j$,
$\wz C^{-1}|Q_{j+1}|\le |Q_j|\le \wz C|Q_{j+1}|$ and
$|Q_j\cap Q_{j+1}|\ge \wz C^{-1}\min\{|Q_j|,\,|Q_{j+1}|\}$
for all $j=0,\,\cdots,\,N-1$.

Let $u\in \dot M^{s,\,p}_\ball(\boz)$ and $g\in \cd^s_\ball(u)$
with $\|g\|_{L^p(\boz)}\ls\|u\|_{\dot M^{s,\,p}(\boz)}$. Then
\begin{eqnarray*}
\int_\boz |u(z)-u_{Q_0}|^{pn/(n-ps)}\,dz&&\ls\sum_{Q\in\cf}\int_Q|u(z)-u_Q|^{pn/(n-ps)}\,dz
+\sum_{Q\in\cf}|Q||u_Q-u_{Q_0}|^{pn/(n-ps)}\\
&&\equiv I_1+I_2.
\end{eqnarray*}
Then by Lemma \ref{l2.4}, $n/(n-ps)> 1$ and the above chain property, we  have
\begin{eqnarray*}
I_1&&\ls\sum_{Q\in\cf}\lf(\int_{2Q}[g(z)]^p\,dz\r)^{n/(n-ps)}\\
&&\ls \lf(\sum_{Q\in\cf}\int_{2Q}[g(z)]^p\,dz\r)^{n/(n-ps)}
\ls \lf(\int_\boz [g(z)]^p\,dz\r)^{n/(n-ps)}
\ls \|u\|_{\dot M^{s,\,p}_\ball(\boz)}^{pn/(n-ps)}.
\end{eqnarray*}

To estimate $I_2$,  for every $Q\in\cf$,
let $\{Q_j\}_{j=1}^N$ be as in (b).
Then we have
\begin{eqnarray*}
|u_Q-u_{Q_0}|
&&\ls \sum_{j=0}^{N-1}(|u_{Q_j}-u_{Q_j\cap Q_{j+1}}|+|u_{Q_j\cap Q_{j+1}}-u_{Q_{j+1}}|)\\
&&\ls \sum_{j=0}^{N}\bint_{Q_j}|u(z)-u_{Q_j}|\,dz \\
&&\ls \sum_{j=0}^{N}|Q_j|^{s/n}\lf(\bint_{2Q_j}[g(z)]^p\,dz\r)^{1/p}\\
&&\ls \sum_{\wz Q\in\cf:\, Q\subset 2\wz Q}|\wz Q|^{s/n}\lf(\bint_{2\wz Q}[g(z)]^p\,dz\r)^{1/p}\\
\end{eqnarray*}
Thus, by $Q\subset \wz C\wz Q$ for all $\wz Q\in\cf$,
we obtain
\begin{eqnarray*}
I_2&&\ls \sum_{Q\in\cf} |Q|
\lf\{\sum_{\wz Q\in\cf:\, Q\subset\wz C\wz Q}|\wz Q|^{s/n}\lf(\bint_{2\wz Q}[g(z)]^p\,dz\r)^{1/p}\r\}^{pn/(n-ps)}\\
&&\ls\sum_{Q\in\cf} \int_Q
\lf\{\sum_{\wz Q\in\cf}|\wz Q|^{s/n}\lf(\bint_{2\wz Q}[g(z)]^p\,dz\r)^{1/p}\chi_{\wz C\wz Q}(x)\r\}^{pn/(n-ps)}\,dx\\
&&\ls\int_\boz
\lf\{\sum_{\wz Q\in\cf}\lf[\cm\lf(\lf[|\wz Q|^{s/n}\lf(\bint_{2\wz Q}[g(z)]^p\,dz\r)^{1/p}\chi_{\wz Q}\r]^{1/2}\r)(x)\r]^2\r\}^{pn/(n-ps)}\,dx,
\end{eqnarray*}
where  $\cm$ denotes the Hardy-Littlewood maximal operator.
Then  by the vector-valued inequality of $\cm$ (see, for example, \cite{s93}), we have
\begin{eqnarray*}
I_2&&\ls\int_\boz
\lf\{\sum_{\wz Q\in\cf}|\wz Q|^{s/n}\lf(\bint_{2\wz Q}[g(z)]^p\,dz\r)^{1/p}\chi_{\wz Q}(x)\r\}^{pn/(n-ps)}\,dx\\
&&\ls \lf(\sum_{\wz Q\in\cf}\int_{2\wz Q}[g(z)]^p\,dz\r)^{n/(n-ps)}\ls \lf(\int_\boz[g(z)]^p\,dz\r)^{n/(n-ps)}\ls \|u\|_{\dot M^{s,\,p}_\ball(\boz)}^{pn/(n-ps)}.
\end{eqnarray*}
This estimate finishes the proof of Theorem \ref{t1.2}(i).

(ii) Assume that $\boz$ is bounded domain and has a separation property with respect to $x_0\in\boz$
and   $C_0\ge1$. For any fixed point $x\in\boz$, let $\gz$ be a curve as
in Definition \ref{d2.3}.
We claim that
$d(\gz(t),\,\boz^\complement)\gs \diam\gz([0,\,t])$ for all $t\in[0,\,1]$.
Assume this claim holds for  the moment.
Then, as pointed out in \cite{bk95},
even though the claim is not enough to ensure that $\gz$ is a John curve for $x$,
it is known that the claim is enough to guarantee that $\gz$ can be
modified to yield a John curve for $x$ by the arguments in \cite[pp.\,385-386]{ms79}
and \cite[pp.\,7-8]{nv91}.

To prove the above claim, let $N=2+C_0/b$,
where $b$ is the constant for which LLC(2) holds.
For $t\in(0,\,1]$,
if $d(\gz(t),\,\boz^\complement)\ge d(x_0,\,\boz^\complement)/N$,
then $$\gz([0,\,t])\subset\boz\subset B\lf(\gz(t),\, \frac{N\diam\boz}{d(x_0,\,\boz^\complement)}d(\gz(t),\,\boz^\complement)\r),$$  which implies the above claim.
Assume that $d(\gz(t),\,\boz^\complement)<d(x_0,\,\boz^\complement)/N$.
Now it suffices to prove  that $\gz([0,\,t])\subset B(\gz(t),\,(N-1)d(\gz(t),\,\boz^\complement))$.
To this end, if  $y\in\gz([0,\,t])\setminus B(\gz(t),\,(N-1)d(\gz(t),\,\boz^\complement))$,
since $d(x_0,\,\gz(t))\ge(N-1)d(\gz(t),\,\boz^\complement)$,
then by Theorem \ref{t1.1}(iii), we know that $x_0$ and $y$ are contained in the same component of
$\boz\setminus B(\gz(t),\,b(N-1)d(\gz(t),\,\boz^\complement))$,
By $b(N-1)> C_0$, we further know that $x_0$ and $y$ are in the same component of
$\boz\setminus \partial B(\gz(t),\,C_0d(\gz(t),\,\boz^\complement))$,
which is a contradiction with the separation property.
This verifies the above claim and thus finishes the proof of Theorem \ref{t1.2}(ii).
\end{proof}

To prove Theorem \ref{t1.3}, we first establish the following result, which is an improvement on
\cite[Theorem 8.7(ii)]{h03}.

\begin{lem}\label{l4.1} Let $s\in(0,\,1]$.
Then there exist positive constants $0<C_1<1<C_2$ such that  for all balls $B\subset\rn$ and
$u\in \dot M^{s,\,n/s}(4B)$,
\begin{equation}\label{e4.1}
 \bint_B\exp\lf(C_1\frac{|u(x)-u_B|}{\|u\|_{\dot M^{s,\,n/s}(4B)}}\r)^{n/(n-s)}\,dx\le C_2.
\end{equation}
\end{lem}

 \begin{proof}

Assume that $B\equiv B(x_0,\,2^{-k_0})$ for some $x_0\in\rn$ and $k_0\in\zz$.
Let $u\in\dot M^{s,\,p}(4B)$ and $g\in\cd^s(u)$ such that
$\|g\|_{L^{n/s}(4B)}\le 2 \|u\|_{\dot M^{s,\,p}(4B)}$.
We extend $g$ to the whole $\rn$ by settting $g(z)=0$ for all $z\in\rn\setminus 4B$.
For every Lebesgue point $x$ of $u$, we have
\begin{eqnarray*}
 |u(x)- u_{B(x ,\,2^{-k_0-1})}|
&&\le \sum_{j\ge k_0+1} |u_{B(x ,\,2^{-j-1})}- u_{B(x ,\,2^{-j})}|+|u_{B(x ,\,2^{-k_0+1})}-u_B|\\
&&\ls \sum_{j\ge k_0+1}  \bint_{B(x ,\,2^{-j})}|u(z)- u_{B(x ,\,2^{-j})}|\,dz\\
&&\ls \sum_{j\ge k_0+1}  \bint_{B(x ,\,2^{-j})}|u(z)- u_{B(x,\,2^{-j+2})\setminus B(x ,\,2^{-j+1})}|\,dz\\
&&\ls \sum_{j\ge k_0+1}  \bint_{B(x ,\,2^{-j})}\bint_{B(x,\,2^{-j+2})\setminus B(x ,\,2^{-j+1})}
|u(z)- u(y)|\,dy\,dz\\
&&\ls \sum_{j\ge k_0+1}  \bint_{B(x ,\,2^{-j})}\int_{B(x,\,2^{-j+2})\setminus B(x ,\,2^{-j+1})}
\frac{|g(z)+ g(y)|}{|y-x|^{n-s}}\,dy\,dz\\
&&\ls \sum_{j\ge k_0+1}\int_{B(x,\,2^{-j+2})\setminus B(x ,\,2^{-j+1})}
\frac{|\cm(g)(y)|}{|y-x|^{n-s}}\,dy\\
&&\ls\int_{B(x_0,\,2^{-k_0+2})}
\frac{|\cm(g)(y)|}{|y-x|^{n-s}}\,dy.
\end{eqnarray*}
Similarly,
\begin{eqnarray*}
 |u_B-u_{B(x ,\,2^{-k_0-1})}|&& \ls \bint_B|u(z)-u_B|\,dz
\ls\int_{B(x_0,\,2^{-k_0+2})}
\frac{|\cm(g)(y)|}{|y-x|^{n-s}}\,dy.
\end{eqnarray*}
 Thus,
$$|u(x)-u_B|\ls \int_{B(x_0,\,2^{-k_0+2})}
\frac{|\cm(g)(y)|}{|y-x|^{n-s}}\,dy.$$
Then  by \cite[Lemma 7.2]{gt}, for all $q\ge n/s$,
$$\|u-u_B\|_{L^q(B)}\le q^{1-s/n+1/q}|B(0,\,1)|^{1-s/n}|B|^{1/q}\|\cm(g)\|_{L^{n/s}(4B)},$$
which together with the $L^{n/s}(\rn)$-boundedness of $\cm$ implies that
$$\bint_B|u(z)-u_B|^q\,dz\ls q^{1+(n-s)/(nq)} |B(0,\,1)|^{nq/(n-s)}\|g\|^q_{L^{n/s}(4B)},$$
and hence for all $q\ge n/s-1$,
$$\bint_B|u(z)-u_B|^{qn/(n-s)}\,dz\ls \frac{nq}{n-s}  \lf(|B(0,\,1)|\frac{nq}{n-s}\|g\|^{n/(n-s)}_{L^{n/s}(4B)} \r)^q.$$
Then taking $\sz> [e|B(0,\,1)|n/(n-s)]^{(n-s)/n}$, we have
$$\bint_{B}\sum_{j\ge \lfloor n/s\rfloor} \frac1{j!}\lf(\frac{|u(x)-u_B|}{\sz\|g\|_{L^{n/s}(4B)}}\r)^{jn/(n-s)}\,dx\ls
 \sum_{j\ge1}\lf(\frac{n|B(0,\,1)|}{(n-s)\sz^{n/(n-s)}}\r)^j\frac{j^j}{(j-1)!}\ls1.$$
Notice that by H\"older inequality, we have
$$\bint_{B}\sum_{j=0}^{\lfloor n/s\rfloor}
\frac1{j!}\lf(\frac{|u(x)-u_B|}{\sz\|g\|_{L^{n/s}(4B)}}\r)^{jn/(n-s)}\,dx
\ls \sum_{j=0}^{\lfloor n/s\rfloor} \lf(\bint_{B}
 \frac{|u(x)-u_B|^{n/s}}{\|g\|^{n/s}_{L^{n/s}(4B)}}\,dx\r)^{ (n-s)/(j-s)}\ls1.$$
This gives \eqref{e4.1} and thus finishes the proof of Lemma \ref{l4.1}.
 \end{proof}

\begin{proof}[Proof of Theorem \ref{t1.3}]

(i) Assume that $\boz$ is a weak carrot domain.
Since $\boz$ is bounded (see \cite[Corollary 1]{ss90}), we may assume that $|\boz|=1$.
Let $$\wz \phi_s(t)=\exp\lf(t^{n/(n-s)}\r)-\sum_{j=0}^{j_0}\frac1{j!}t^{jn/(n-s)}$$
for $t\in(0,\,\fz)$, where $j_0$ denotes the maximal integer no more than $n/s-1$.
Since $\wz \phi_s(t)\sim \phi_s(t)$ for $t\ge 1$, so we only need to prove
\eqref{e1.3} for $\wz \phi_s$; see \cite{a75}.
It further suffices to prove that
there exists a $\sz\in(0,\,1)$ such that for all
  $u\in \dot M^{s,\,n/s}_\ball(\boz)$ with $\|u\|_{\dot M^{s,\,n/s}_\ball(\boz)}=1$,
 \begin{equation}\label{e4.2}
  \int_\boz \wz\phi_s\lf(\sz|u(x)-u_\boz|\r)\,dx\le1.
 \end{equation}

Let $u\in \dot M^{s,\,n/s}_\ball(\boz)$ with $\|u\|_{\dot M^{s,\,n/s}_\ball(\boz)}=1$
and write
\begin{eqnarray*}
\int_\boz \wz\phi_s\lf(\sz|u(x)-u_\boz|\r)\,dx&&= \sum_{j=j_0+1}\frac1{j!}   \sz  ^{jn/(n-s)}
\int_\boz |u(x)-u_\boz|^{jn/(n-s)}\,dx.
\end{eqnarray*}
Since $\boz\subset\cup_{z\in\boz}B(z,\,d(z,\,\boz^\complement)/10)$, then by the standard $1/5$-covering theorem,
there exist points $\{z_i\}_i\subset\boz$ such that $\{B(z_i,\,d(z,\,\boz^\complement)/50)\}_i$ are pairwise disjoint,
  $$1\le \sum_{i}\chi_{B(z_i,\,d(z_i,\,\boz^\complement)/10)}\le \wz C\chi_\boz$$ for some
fixed positive constant $\wz C$. Denote by $W$ the set of balls
$\{B(z_i,\,d(z_i,\,\boz^\complement)/10)\}$.
Let $B_0$ be a fixed ball in $W$ with the largest radius. Denote by $x_B$ the center of ball $B$
and specially $x_0$ of $B_0$.
Then
\begin{eqnarray*}
&& \int_\boz|u(x)-u_\boz|^{jn/(n-s)}\,dx \\
&&\quad\le 2^{jn/(n-s)} \int_\boz|u(x)-u_{B_0}|^{jn/(n-s)}\,dx \\
&&\quad\le 2^{jn/(n-s)} \sum_{B\in W}\int_B|u(x)-u_{B_0}|^{jn/(n-s)}\,dx \\
&&\quad\le 4^{jn/(n-s)} \sum_{B\in W}\int_B|u(x)-u_{B}|^{jn/(n-s)}\,dx
+4^{jn/(n-s)}\sum_{B\in W }|B||u_B-u_{B_0}|^{jn/(n-s)}\\
&&\quad\equiv I_1(j)+I_2(j).
\end{eqnarray*}

Observe that $\sum_{B\in W}\|u\|^{n/s}_{\dot M^{s,\,n/s}(4B)}\le \wz C \|u\|_{\dot M^{s,\,p}(\boz)}\le \wz C$.
Choose $0<\sz< C_1/8$ such that
$$C_2\wz C[(4\sz)(C_1)^{-1}]^{(j_0+1)n/(n-s)}\le 1/2,$$
where $C_1$ and $C_2$ are the constants from Lemma \ref{l4.1}.
Then  by $|\boz|=1$ and Lemma \ref{l4.1},
we have
\begin{eqnarray*}
&&\sum_{j=j_0+1}^\fz \frac1{j!}   \sz  ^{jn/(n-s)} I_1(j)\\
&& \quad
\le \sum_{B\in W} [4\sz(C_1)^{-1}]^{(j_0+1)n/(n-s)} \int_B \sum_{j=j_0+1}  \frac1{j!} (C_1)^{jn/(n-s)}|u(x)-u_{B}|^{jn/(n-s)}\,dx\\
&&\quad\le \sum_{B\in W}[\|u\|_{\dot M^{s,\,n/s}(4B)}(4\sz)(C_1)^{-1}]^{(j_0+1)n/(n-s)}\int_B \exp\lf(C_1\frac{|u(x)-u_B|}{\|u\|_{\dot M^{s,\,n/s}(4B)}}\r)^{n/(n-s)}\,dx\\
&&\quad\le C_2[(4\sz)(C_1)^{-1}]^{(j_0+1)n/(n-s)} \sum_{B\in W}\|u\|^{n/s}_{\dot M^{s,\,n/s}(4B)}\le1/2.
\end{eqnarray*}

To estimate $I_2(j)$, for each $B\in W\setminus\{B_0\}$,
let $\gz$ be the geodesic joining $x_0$ and $x_B$.
By using the Bescovitch covering lemma (see \cite{s93}) and some arguments similar to
these in the proofs of \cite[Theorem 4.1] {bk96} and \cite[Lemma 3.2]{s10},
we can find a family of balls $\cb\equiv\{B_i\}_{i=0}^N$ such that

 a) $B_i\equiv B(w_i,\,d(w_i,\,\boz^\complement)/10)$
with $w_i\in\gz$ for all  $i=0,\,\cdots,\,N$, $w_0=x_0$ and $w_N=x_B$;

 b) $B_i\cap B_{i+1}\ne\emptyset$ for all $i=0,\,\cdots,\,N-1$;

 c) $\sum_{i=1}^N\chi_{2B_i}(z)\le \overline C$
for all $z\in\boz$, where the constant $\overline C$ only depends on the dimension $n$.

Let $\wz w_i \in B_i\cap B_{i+1}$ and $\wz r_i\equiv \min\{d(w_i,\,\boz^\complement),\,d(w_{i+1},\,\boz^\complement)\}$
 for all $i=0,\,\cdots,\,N-1$.
Notice that b) implies that
$$\frac 9{11}d(w_{i+1},\,\boz^\complement)\le d(w_i,\,\boz^\complement)\le \frac{11}9d({w_i+1},\,\boz^\complement)$$
for all $i=0,\,\cdots,\,N-1$.
Since $d(z,\,\boz^\complement)\sim  d(w_i,\,\boz^\complement)$ for all $z\in B_i$,
so by c), we have
 $$N\ls \sum_{i=0}^N\int_{B_i\cap \gz}\frac1{d(z,\,\boz^\complement)}\,|dz|\ls k_\boz(x_B,\,x_0).$$
Thus, by these, c) and the H\"older inequality, we have
\begin{eqnarray*}
|u_B-u_{B_0}|&&\ls\sum_{j=0}^{N-1}|u_{B_i}-u_{B(\wz w_i,\,\wz r_i)}|+|u_{B(\wz w_i,\,\wz r_i)}- u_{B_{i+1}}|\\
&&\ls\sum_{j=0}^N\bint_{2B_i}|u(z) -u_{2B_i}|\,dz\\
&&\ls \sum_{j=0}^N [d(w_i,\,\boz^\complement)]^{-n+s }\int_{2B_i} g(x)\,dx\\
&&\ls \sum_{j=0}^N \lf\{
\int_{2B_i} [g(x)]^{n/s}\,dx\r\}^{s/n}\ls N^{(n-s)/n}\ls [k_\boz(x_B,\,x_0)]^{ (n-s)/n}.
\end{eqnarray*}
Let $C_3$ be the constant from the preceding inequality and
notice that there exists a positive constant $C_4$ such that for all $x\in B$ and $B\in W\setminus\{B_0\}$,
$k_\boz(x_B,\,x_0)\le C_4 k_\boz(x,\,x_0)$. By $|\boz|=1$, we have
\begin{eqnarray*}
 \sum_{j=j_0+1}^\fz \frac1{j!}\sz^{jn/(n-s)} I_2(j)
&&\le  \sum_{j=j_0+1}^\fz \frac1{j!}(4\sz C_3)^{jn/(n-s)}
 \sum_{j=j_0+1}^\fz\sum_{B\in W\setminus\{B_0\}}|B|[k_\boz(x_B,\,x_0)]^{j (n-s)/n}\\
&&\le\sum_{j=j_0+1}^\fz \frac1{j!}  (4\sz C_3C_4)^{jn/(n-s)} \int_{\boz\setminus B_0}  [k_\boz(x,\,x_0)]^{j  (n-s)/n}\,dx\\
&&\le \sum_{j=j_0+1}\frac1{j!} (4\sz C_3C_4)^{jn/(n-s)}  \sz  ^{jn/(n-s)}\int_\boz [k_\boz(x_0,\,x)]^j\,dx\\
&&\le \int_\boz [\exp((4\sz C_3C_4)^{n/(n-s)}k_\boz(x_0,\,x))-1]\,dx.
\end{eqnarray*}
Then  by Lemma \ref{l2.1}, we can choose $\sz$ small enough such that
$\sum_{j=j_0+1}^\fz \frac1{j!}\sz^{jn/(n-s)} I_2(j) \le 1/2,$
which together with
$\sum_{j=j_0+1}^\fz \frac1{j!}\sz^{jn/(n-s)} I_1(j) \le 1/2 $ implies \eqref{e4.2}.
This finishes the proof of Theorem \ref{t1.3}(i).

(ii) Assume that $\boz$ has the slice property with respect to  $y$ and $C_5\ge1$
as in Definition \ref{d2.4}.
Then for every $x\in\boz$, by Definition \ref{d2.4}, there exist a rectifiable curve $\gz$  and
a sequence of $\{S_i\}_{i=0}^j$ for some $j\ge0$, satisfying (i) through (iv) of Definition \ref{d2.4}.
Without loss of generality, we may assume that   $j\ge2$.
In fact, as pointed out  by Buckley and Koskela \cite[p.\,890]{bk96},
  Definition \ref{d2.4}(iii) and (iv)
implies that  $j+1\ge k_\boz(x,\,y)/C_5$.
Observe that if $k(x,\,y)\le 2C_5$, then \eqref{e2.1} is clearly satisfied.
So we only need to consider the case $j\ge2$.

For each $i=1,\,\cdots,\,j-1$, define the function $u_i$ by setting
$u_i(z)\equiv\inf_{\wz\gz}\ell(\wz\gz\cap S_i)$ for all $z\in\boz$,
where the infimum is taken over all rectifiable curves $\wz \gz$ joining $x$ and $z$.
Obviously, $u_i(z)=0$ for $z\in\cup_{k=0}^{i-1}S_k$ and $u_i(z)$ is a constant for $z\in\cup_{k=i+1}^{j}S_k$.
Then  by an argument similar to the one in the proof of
 Lemma \ref{l3.2},
we can prove that $u\in\dot M^{s,\,n/s}_\ball(\boz)$ and
$g_i\equiv r_i^{1-s}\chi_{B(x_i,\,2C_5d(x_i,\,\boz^\complement))}$ is a constant multiple of an element of
$\cd^{s,\,1/(16C_5)}_\ball(u)$,
where $r_i=\diam S_i\sim d(x_i,\,\boz^\complement) $ by Definition \ref{d2.4}(iii).
We omit the details.
Then
$\|u_i\|_{\dot M^{s,\,p}_\ball(\boz)}\ls r_i^{1-s+n/p}$.
Notice that   Definition \ref{d2.4}(ii) implies that $|u_i(x)-u_i(y)|\ge C_5^{-1}d(x_i,\,\boz^\complement)\gs r_i$.
Moreover, define
$u=j^{-s/n}\sum_{i=1}^{j-1} r_i^{-1}u_i$.
Then the function
$$g\equiv j^{-s/n}\sum_{i=1}^{j-1} r_i^{-1}g_i=j^{-s/n}\sum_{i=1}^{j-1} r_i^{-s}
\chi_{B(x_i,\,2C_5d(x_i,\,\boz^\complement))}$$ is a constant multiple of an element of
$\cd^{s,\,1/(16C_5)}_\ball(u)$, which together with the Fefferman-Stein vector-valued   inequality of the
Hardy-Littlewood maximal function $\cm$ (see, for example, \cite{s93}) and Definition \ref{d2.4} imply that
\begin{eqnarray}\label{e4.3}
 \|u\|^{n/s}_{\dot M^{s,\,n/s}_\ball(\boz)}&&\ls j^{-1}
\int_\boz\lf(\sum_{i=1}^{j-1}r_i^{-s}
\chi_{B(x_i,\,2C_5d(x_i,\,\boz^\complement))}(z)\r)^{n/s}\,dz\\
&&\ls
j^{-1}
\int_\boz\lf(\sum_{i=1}^{j-1}
\lf[\cm\lf(\lf[r_i^{-s}\chi_{B(x_i,\,C_5^{-1}d(x_i,\,\boz^\complement))}\r]^{1/2}\r)(z)\r]^2\r)^{n/s}\,dz\nonumber\\
&&\ls j^{-1}
\int_\boz\lf(\sum_{i=1}^{j-1}r_i^{-s}
 \chi_{B(x_i,\,C_5^{-1}d(x_i,\,\boz^\complement))} (z)\r)^{n/s}\,dz\nonumber\\
&&\ls j^{-1}
\int_\boz \sum_{i=1}^{j-1}r_i^{-n}
 \chi_{B(x_i,\,C_5^{-1}d(x_i,\,\boz^\complement))} (z) \,dz\ls 1.\nonumber
\end{eqnarray}
On  the other hand,  since
$u(z)\ge (j-1)^{1-s/n}$ for $z\in S_j$,
then for $t^{n/(n-s)}\le  (j-1)/\log(1+[C_5 ^{-1}d(x,\,\boz^\complement)]^{-n})$, we have
$$
\int_\boz \phi_s\lf(\frac{u(z)}{t}\r)\,dz\ge\int_{S_j} \phi_s\lf(\frac{u(z)}{t}\r)\,dz
\ge  [C_5^{-1}d(x,\,\boz^\complement)]^{-n}|S_j|>1,$$ which implies that
$$\|u\|_{\phi_s(L)(\boz)}\gs\lf\{\frac {j-1}{\log(1+[C_5^{-1}d(x,\,\boz^\complement)]^{-n})}\r\}^{(n-s)/n}\gs
\lf\{\frac {j }{\log(1+[ d(x,\,\boz^\complement)]^{-1})}\r\}^{(n-s)/n}.$$
From this, \eqref{e4.3} and $s$-Trudinger inequality, it follows that
 $j\ls{\log(1+[d(x,\,\boz^\complement)]^{-1})}$, which further implies that
$$\int_\gz\frac1{d(z,\,\boz^\complement)}\,|dz|\ls\log(1+[d(x,\,\boz^\complement)]^{-1}).$$
This means that $\boz$ is a weak carrot domain and thus
finishes the proof of Theorem \ref{t1.3}(ii).
\end{proof}

{\bf Acknowledgements.}
The author would like to thank Professor Pekka Koskela 
for his kind suggestions and significant discussions on this topic
and also thank Professor Dachun Yang for his kind suggestions.

\bigskip

\noindent Yuan Zhou

\noindent Department of Mathematics and Statistics,
P. O. Box 35 (MaD),
FI-40014, University of Jyv\"askyl\"a,
Finland

\smallskip

\noindent{\it E-mail address}:  \texttt{yuzhou@cc.jyu.fi}

\end{document}